\documentclass[10pt]{amsart}
\usepackage{amsmath}
\usepackage{amsfonts}
\usepackage{amssymb}
\usepackage{hyperref}
\usepackage{enumerate}
\usepackage{enumitem}
\setlist[enumerate,1]{label={\upshape(\roman*)}}

\newtheorem{thm}{Theorem}[section]

\newtheorem{lemma}[thm]{Lemma}
\newtheorem{prop}[thm]{Proposition}

\theoremstyle{definition}
\newtheorem{defn}[thm]{Definition}

\newtheorem{exam}[thm]{Example}

\theoremstyle{remark}
\newtheorem{remark}[thm]{Remark}

\numberwithin{equation}{section}

\newcommand{\bb}[1]{\mathbb{#1}}
\newcommand{\cl}[1]{\mathcal{#1}}

\newcommand{\inner}[2]{\left\langle {#1},{#2} \right\rangle}
\newcommand{\inductlim}[1]{\underset{\rightarrow \textbf{#1}}{\lim}}
\newcommand{\projectlim}[1]{\underset{\leftarrow \textbf{#1}}{\lim}}

\usepackage[all]{xy}

\begin{document}
\title[Matrix-ordered duals and projective limits]{Matrix-ordered duals of separable operator systems and projective limits}

\author[Wai Hin ~Ng]{Wai Hin~Ng}
\address{Department of Applied Mathematics, The Hong Kong Polytechnic University,
Hong Kong, China}
\email{ricky.wh.ng@polyu.edu.hk}

\date{ March 5, 2018 \,(Last revised)}
\begin{abstract}
We construct projective limit of projective sequence in the following categories: Archimedean order unit spaces with unital positive maps and operator systems with unital completely positive maps. We prove that inductive limit and projective limit in these categories are in duality, provided that the dual objects remain in the same categories and the maps are order embeddings. We generalize a result of Choi and Effros \cite{CE} on matrix-ordered duals of finite-dimensional operator systems to separable case. 
\end{abstract}

\maketitle

\section{Introduction}\label{Section:Intro}
In the past decade, the theory of operator systems has drawn a fair amount of attention in non-commutative functional analysis and quantum information theory. Among some of these works, the matrix-ordered duals of operator systems play an important role in tensor product theory \cite{FP, Ka, KPTT1}, as well as quantum graph theory  \cite{FKPT, PT2}. 
As a result of Choi-Effros abstract characterization of operator systems \cite[Theorem 4.4]{CE}, the matrix-ordered dual $S'$ of a finite-dimensional operator system $S$ remains to be an operator system with a suitable choice of order unit. 
\begin{thm}[Choi-Effros \cite{CE}]\label{DualFiniteDimension}
Let $S$ be a finite-dimensional operator system. Then 
	\begin{enumerate}[label={\upshape(\roman*)}, align=left, widest=iii, leftmargin=*]
		\item there exists faithful $f \colon S \to \bb{C}$; and
		\item any faithful functional $f$ is an Archimedean matrix order unit for $S'$.
	\end{enumerate}
Consequently, $S'$ with any faithful state is an operator system.  
\end{thm}

For instance, if $S \subset M_n$ is an operator system, then $S'$ with the trace functional is an operator system. 
In general, when $\dim(S) = \infty$, it is not clear whether $S'$ possesses an Archimedean matrix order unit.

A natural attempt to this question is approximation by inductive sequence of finite-dimensional operator systems. Inductive limits of complete operator systems were introduced by Kirchberg in \cite{K1995CAR}, which relies on the norm structure.  Recently, inductive limits of (non-complete) operator systems are also studied \cite{ Li2017inductive, LK2016inductive}; and a systematic investigation has started in \cite{MT}. In particular, in \cite[\S 3.2]{MT} Mawhinney and Todorov showed that every inductive sequence of operator systems, via duality, induces a projective sequence of the corresponding state spaces, whose projective limit is homeomorphic to the state space of its inductive limit. 

Motivated by their work, this paper aims to provide a construction of projective limits in the categories of AOU spaces and operator systems. We are also interested in the duality between inductive and projective limits when the dual objects remain in the same categories. As a main application, we generalize Theorem \ref{DualFiniteDimension} to separable operator systems; see Theorem \ref{DualSeparable}. 

The paper is organized as follows. We begin with the preliminaries in \S \ref{Section:Preliminaries}. These include the basics of AOU spaces, operator systems, and inductive limits in these categories developed in \cite{PT, PTT, MT} respectively. 
In \S \ref{Section:Projective limits}, we construct projective limit in AOU spaces and operator systems using their order structures and the corresponding order norms. 
Then we show that inductive and projective limits are in duality, provided the dual objects remain in the same categories and the maps are unital complete order embeddings in \S \ref{Section:Duality injective}. In \S \ref{Section:dualityMain}, we show the existence of faithful state for separable operator systems, which allows us to generalize Theorem \ref{DualFiniteDimension} using inductive and projective limits. 

\section{Preliminaries}\label{Section:Preliminaries}
We outline the basics of Archimedean order unit vector spaces and operator systems developed by Paulsen and Tomforde in \cite{PT}. We also summarize some results of \cite{MT} that will be used in  \S \ref{Section:Duality injective}. 

\subsection{AOU spaces}\label{AOUspaces}
A \textit{$\ast$-vector space} $V$ is a complex vector space equipped with an involution $\ast$. Elements in the real subspace $V_h = \{v \in V \colon v^* = v \}$ are called Hermitian or self-adjoint. Note that for $v \in V$, $v = \Re(v) + i \Im(v)$, where $\Re(v) = \frac{1}{2}(v+v^*)$ and $\Im(v) = \frac{1}{2i} (v - v^*)$ are Hermitian. 

An \textit{ordered $\ast$-vector space} is a pair $(V, V^+)$, where $V$ is a $\ast$-vector space and $V^+$ is a proper cone in $V_h$. The cone induces a natural partial ordering on $V_h$, i.e. $v \leq w$ if and only if $w - v \in V^+$. 
Given an ordered $\ast$-vector space $(V, V^+)$, an element $e \in V^+$ is called an \textit{order unit} if for each $v \in V_h$, there exists $r \geq 0$ such that $v \leq re$. The triple $(V, V^+, e)$ is called an order unit space. 
The order unit $e$ is called \textit{Archimedean}, provided that for all $v \in V_h$,
	\begin{equation*}
		 \forall r > 0, re + v \in V^+ \Longrightarrow v \in V^+.
	\end{equation*}
In this case, the triple $(V, V^+, e)$ is called an \textit{Archimedean order unit space}, or an AOU space for short. We often write denote it by $(V, e)$ or simply $V$ whenever the content is clear.

Given two AOU spaces $(V, V^+, e_V)$ and $(W, W^+, e_W)$, a linear map $\phi \colon V \to W$ is called \textit{positive} if $\phi(V^+) \subset W^+$ and \textit{unital} if $\phi(e_V) = e_W$. It is an \textit{order isomorphism} provided it is bijective and $\phi(v) \in W^+$ if and only if $v\in V^+$. A \textit{state} on $V$ is a unital positive functional $f \colon V \to \mathbb{C}$. We write $\mathfrak{S}(V)$ for the set of states on $V$ and call it the \textit{state space} of $V$. Note that $\mathfrak{S}(V)$ is a cone in the algebraic dual of $V$.

Given an order unit space $(V, V^+, e)$, there is a seminorm of $V_h$ by
	\begin{equation}\label{OrderNorm}
		|| v ||_h = \inf \{ r > 0 \colon - re \leq v \leq re \}.
	\end{equation}
We call $|| \cdot ||_h$ the \textit{order seminorm} on $V_h$ determined by $e$. By \cite[Theorem 2.30]{PT}, $e$ is Archimedean if and only if $V^+$ is closed in $V_h$ in the order topology induced by $|| \cdot ||_h$. In this case, by \cite[Proposition 2.23]{PT} $||\cdot||_h$ is a norm on $V_h$, and we call $|| \cdot ||_h$ the \textit{order norm} determined by $e$ on $V_h$. A \textit{$\ast$-seminorm} $|| \cdot ||$ on $V$ is a seminorm such that $||v^*|| = || v ||$; it is called an \textit{order seminorm} if $|| v || = || v ||_h$ for every $v \in V_h$. In \cite{PT}, Paulsen and Tomforde introduced the \textit{minimal order seminorm} 
	\begin{equation}
		|| v ||_m  	:= \sup \{|f(v)| \colon f \in \mathfrak{S}(V)\},	
	\end{equation}				 
and the \textit{maximal order seminorm}
	\begin{equation}
		|| v ||_M 	:= \inf \{ \sum_{i=1}^n | \lambda_i | ||v_i||_h \colon v = \sum_{i=1}^n \lambda_i v_i, \; \lambda_i \in \mathbb{C}, v_i \in V_h\};
	\end{equation}
and $|| \cdot ||_m \leq || \cdot || \leq || \cdot ||_M \leq 2|| \cdot ||_m$ for all order seminorms \cite[Proposition 4.9]{PT}. Moreover, if $(V, V^+ e)$ is an AOU space, then every order seminorm is an order norm. 
The following proposition is part of \cite[Theorem 4.22]{PT}.

\begin{prop}\label{upContractive}
A unital linear map $\phi$ between AOU spaces is positive if and only if it has norm one with respect to the minimal order seminorms of both spaces.
\end{prop} 

We remark that the original statement also takes into account of the \textit{decomposition order seminorm}, which is not used in this paper. We have a partial converse for the maximal order seminorms.

\begin{lemma}
Let $\phi \colon V \to W$ be a unital positive map between real AOU spaces. Then $\phi$ is contractive with respect to the order seminorms. 
\end{lemma}

\begin{proof}
For each $v \in V$, let $r = || v ||$, then $r \phi(v) \pm e_W = \phi( rv \pm e_V ) \geq 0$. By definition of order seminorm, $|| \phi(v) || \leq || v ||$. 
\end{proof}

\begin{prop}
Let $\phi \colon V \to W$ be unital positive map between AOU spaces. Then $\phi$ is contractive with respect to the maximal order norms on $V$ and $W$; in fact, $|| \phi ||_M := \sup \{ ||\phi(v)||_M \colon || v ||_M \leq 1  \} = 1$.
\end{prop}

\begin{proof}
Let $v \in V$ and write $v = \sum_{i=1}^n \lambda_i v_i$, where $\lambda_i \in \mathbb{C}$ and $v_i \in V_h$. Then
	\begin{align*}
		|| \phi(v) ||_M 	&\leq 	\sum_{i=1}^n | \lambda_i | || \phi(v_i) ||_h 	\leq \sum_{i=1}^n |\lambda_i| || v_i ||_h,
	\end{align*}
where the last inequality follows from the previous lemma. By taking the infinmum over all such representations $v = \sum_i \lambda_i v_i$, we deduce that $|| \phi(v) ||_M \leq || v ||_M$. Also, $|| \phi(e_V) ||_M = || e_W ||_M = 1$, hence $|| \phi ||_M = 1$. 
\end{proof}

The partial ordering on the order unit space $(V, V^+, e)$ gives rise to an \textit{order topology}, which by \cite[Proposition 4.9]{PT}, is equivalent to the \textit{seminorm topology} induced by any of the order seminorms. Moreover, the subspace topology on $V_h$ is equivalent to the topology induced by $|| \cdot ||_h$ on $V_h$. By Proposition \ref{upContractive}, unital positive map $\phi \colon V \to W$ between AOU spaces is continuous with respect to the order topology. 

We denote $V'$ the space of continuous linear functionals in the order topology. We let $\phi' \colon W' \to V'$, $\phi'(f) = f \circ \phi$, be the \textit{dual map} or \textit{adjoint} of $\phi$. When $V$ is an AOU space, $V'$ is the dual normed space with respect to any of the order norm on $V$, hence it is a Banach space. We equip $V'$ the weak*-topology generated by the order norm topology on $V$. By \cite[Theorem 5.2]{PT}, the state space $\mathfrak{S}(V)$ is a compact cone that spans $V'$.

\begin{defn}[Ordered dual]
Given an AOU space $(V, V^+, e)$, we define an involution on $V'$ by $f^*(v) := \overline{f(v)}$. We equip $V'$ the natural order $f \in (V')^+$ if and only if $f$ is a positive linear functional. We call the ordered $\ast$-vector space $(V', (V')^+)$ the \textbf{ordered dual} of $V$ and denote it by $V'$. Note that the ordered dual of an AOU space need not be an AOU space.
\end{defn}

We denote by \textbf{OU} the category whose objects are order unit spaces with morphisms being unital positive maps, and by \textbf{AOU} the category whose objects are AOU spaces with the same morphisms. The process of Archimedeanization is a functor from \textbf{OU} to \textbf{AOU} by forming some quotient of $V$ and taking closure of $V^+$, see \cite[\S 3.2]{PT}.

\subsection{Operator Systems}\label{OS}
Given a $\ast$-vector space $S$, for each $n \in \mathbb{N}$, we identify the vector space tensor product $M_n \otimes S = M_n(S)$, whose elements are $n$ by $n$ matrices with entries in $S$, equipped with the involution $[s_{ij}]^* := [s_{ji}^*]$. It follows that $M_n(S)$ is a $\ast$-vector space, and we denote $M_n(S)_h$ for its Hermitian subspace. A \textit{matrix ordering} on $S$ is a family of cones $C_n \subset M_n(S)_h$, $n \in \bb{N}$, satisfying: 
	\begin{enumerate}[label={\upshape(\roman*)}, align=left, widest=iii, leftmargin=*]
		\item 	$C_n \cap -C_n = \{ 0 \}$, for all $n \geq 1$;
		\item 	$M_n(C_n)$ is the complex span of $C_n$;
		\item 	$\alpha^* C_n \alpha \subset C_m$, for each $\alpha \in M_{n,m}(\mathbb{C})$.
	\end{enumerate}
This last condition is often called \textit{compatibility} of $\{C_n\}$. 
A \textit{matrix-ordered $\ast$-vector space} is a pair $(S, \{C_n\}_{n=1}^{\infty})$, where $S$ is a $\ast$-vector space and $\{C_n\}_{n=1}^{\infty}$ is a matrix ordering. When the content is clear, we often write $M_n(S)^+$ for $C_n$. Note that in this case, for each $n\in\mathbb{N}$, $(M_n(S), M_n(S)^+)$ is a $\ast$-ordered vector space. 

An element $e \in C_1 = S^+$ is a \textit{matrix order unit}, provided that $I_n \otimes e$ is an order unit for $(M_n(S), M_n(S)^+)$ for every $n \in \mathbb{N}$; it is called an \textit{Archimedean matrix order unit} if $I_n \otimes e$ is an Archimedean order unit for $(M_n(S), M_n(S)^+)$ for each $n \in \mathbb{N}$. 
The triple $(S, \{C_n\}_{n=1}^{\infty}, e)$ is called a \textit{matrix-ordered $\ast$-vector space with a matrix order unit}, or MOU space for short, provided $e$ is a matrix order unit; it is called an \textit{abstract operator system} if $e$ is an Archimedean matrix order unit. 
We often denote it by the triple $(S, \{C_n\} , e)$, $(S, e)$, or simply $S$ whenever the content is clear.  

Let $\phi \colon S \to T$ be a linear map between MOU spaces $S$ and $T$. For each $n \in \mathbb{N}$, we write $\phi^{(n)} = id_n \otimes \phi \colon M_n(S) \to M_n(T)$  by $A \otimes s \mapsto A \otimes \phi(s)$. We say that $\phi$ is \textit{$n$-positive} if $\phi^{(n)}$ is positive between the order unit spaces $M_n(S)$ and $M_n(T)$; and $\phi$ is \textit{completely positive} provided that $\phi$ is $n$-positive for every $n \in \mathbb{N}$. We write $CP(S, T)$ (resp. $UCP(S, T)$) for the cone of (resp. unital) completely positive maps from $S$ to $T$. We say that $\phi$ is a \textit{complete order isomophism} if $\phi$ is bijective and both $\phi$ and $\phi^{-1}$ are completely positive; $\phi$ is a \textit{complete order embedding} if $\phi$ is a complete order isomorphism onto its range. We denote by \textbf{MOU} the category whose objects are MOU spaces with morphisms being unital completely positive maps, and by \textbf{OS} the category whose objects are operator systems with the same morphisms.  The process of Archimedeanization from \textbf{MOU} to \textbf{OS} was explicitly studied in \cite[\S 3.1]{PTT}. 

A \textit{concrete operator system} is a unital selfadjoint subspace $S$ of $B(\cl{H})$, the C*-algebra of bounded linear operators on a Hilbert space $\cl{H}$. Naturally $S$ inherits a matrix ordering $\{M_n(S)^+\}_{n=1}^{\infty}$ from $B(\cl{H})$ by $M_n(S)^+ := M_n(B(\cl{H}))^+ \cap M_n(S)$, where we identify $M_n(B(\cl{H})) \cong B( \oplus_{i=1}^n \cl{H} )$. Moreover, the identity $I$ is an Archimedean matrix order unit for $(S, \{M_n(S)^+\}_{n=1}^{\infty})$, thus it is an abstract operator system. The converse was proved by Choi and Effros \cite[Theorem 4.4]{CE}.

The next proposition can be found in \cite[Remark 1.2]{FNT2017}, which is a property of MOU spaces rather than operator systems. We include their proof for completeness as this handy tool reduces the complexity in many proofs in the literature. 

\begin{prop}\label{MOUunit}
Let $S$ be a matrix-ordered $\ast$-vector space. Then $e \in S^+$ is an order unit if and only if it is a matrix order unit. 
\end{prop}

\begin{proof}
Given $A \in M_n(S)_h$, decompose $A = \sum_i A_i \otimes x_i$, where $A_i \in (M_n)_h$ and $x_i \in S_h$ by  \cite[Lemma 3.7]{PTT}. Since $e$ is an order unit, there exists $r > 0 $ such that $r e \pm x_i \in S^+$ for each $i$. Decompose $A_i = P_i - Q_i$, where $P_i, Q_i \in M_n^+$. Note that
	\begin{align*}
		r (\sum_i P_i + Q_i) \otimes e - A 	&= \sum_i P_i \otimes (re - x_i) + \sum_i Q_i \otimes (re + x_i)
	\end{align*}
is in $M_n(S)^+$. By choosing $\lambda > 0 $ such that $\lambda I_n \geq r \sum_i P_i + Q_i$, we deduce that $\lambda I_n \otimes e - A \in M_n(S)^+$. Therefore, $e$ is a matrix order unit for $S$. 
\end{proof}

We also need the following lemma on norm bound. We write $|| T ||_{op}$ for the operator norm of an operator $T$ over a Hilbert space $\mathcal{H}$. 

\begin{lemma}\label{op-norm_estimate}
Let $S \subset B(\mathcal{H})$ be a concrete operator system. Then for each $n \in \mathbb{N}$ and $[T_{ij}] \in M_n(S)$, $|| [T_{ij}] ||_{op} \leq n \cdot \max_{ij} || T_{ij} ||_{M}$.  
\end{lemma}

\begin{proof}
For every $T_{ij} \in S \subset B(H)$, a direct calculation shows that $|| [T_{ij}] ||_{op} \leq ( \sum_{i,j=1}^n || T_{ij} ||_{op}^2 )^{1/2}$. Since the operator norm on $S$ is also an order norm, $|| T_{ij} ||_{op} \leq || T_{ij} ||_M$ and the result follows. 
\end{proof}

\subsection{Matrix-ordered duals of operator systems}\label{Section:dualopsys}
Given an operator system $(S, \{M_n(S)^+\}_{n=1}^\infty, e)$, its underlying space is the AOU space $(S, S^+, e)$ of which the ordered dual is $S' = (S', (S')^+)$. A positive linear functional $f$ on $M_n(S)$ is identified to the matrix $[f_{ij}] \in M_n(S')$, where $f_{ij}(x) := f( E_{ij} \otimes x)$.  
By \cite[Theorem 4.3]{PTT}, this identification endows $S'$ a canonical matrix ordering:
\begin{equation*}
[f_{ij}] \in M_n(S')^+ \overset{\text{def}}{\iff} f( [v_{ij}] ) = \sum_{ij} f_{ij}(v_{ij}) \; \text{is positive.}
\end{equation*} 
 
On the other hand, a functional $f$ on $M_n(S)$ can be identified to $F \colon S \to M_n$ via $F(x) := [ f(E_{ij} \otimes x) ] = [f_{ij}(x)]$. By \cite[Theorem 6.1]{Pa2}, $f$ is positive if and only if $F$ is $n$-positive, if and only if $F$ is completely positive. Therefore, we obtain the following identification:
\begin{equation}\label{DualMatrixOrder}
[f_{ij}] \in M_n(S')^+ \overset{\text{def}}{\iff} F(x) = [f_{ij}(x)] \; \text{is completely positive.}
\end{equation}
We simply denote it by $M_n(S')^+ \cong CP(S, M_n)$. The identification $f \longleftrightarrow F$ in \cite[Chapter 6]{Pa2} has a factor of $n$ and $\frac{1}{n}$. We omitted it as it does not affect complete positivity. 

\begin{defn}[Matrix-ordered dual]\label{OSdual}
We call this matrix-ordered $\ast$-vector space $(S', \{ M_n(S')^+ \}_{n=1}^{\infty})$ the \textbf{matrix-ordered dual} of $S$, and simply denote it by $S'$ whenever the content is clear. 

Following the discussion after \cite[Theorem 4.3]{PTT}, the identification $f \longleftrightarrow F$ asserts that the weak*-topology on $S'$ endows $M_n(S')$ a topology that is equivalent to the weak*-topology on $M_n(S)'$. We call this topology, unambiguously, the \textit{weak*-topology} on $M_n(S')$.
\end{defn}

\begin{remark}\label{Remark:ordered_dual}
Some authors take the matrix-ordered dual to be the \textit{algebraic dual} $S^d$ of $S$, equipped with the same cone $M_n(S^d)^+ \cong CP(S, M_n)$. However, by \cite[Lemma 4.2]{PTT}, its complex span is again $M_n(S')$, so there is no loss of generality to replace $S^d$ with $S'$, which already has the weak*-topology.
The fact that $S'$ and $M_n(S')$ are topological vector spaces turns out to be crucial in \S \ref{Section:Projective limits}. 

Also, recall that given an operator space $V$, the \textit{operator space dual} is the underlying space $V'$ equipped with the operator space structure given by $M_n(V') \cong CB(V, M_n)$ complete norm isometrically, see \cite{BP, ER}. In a similar vein, for operator system $S$, we have $M_n(S')^+ \cong CP(S, M_n)$, complete order isomorphically. The Wittstock's decomposition theorem \cite[Theorem 8.5]{Pa2} asserts that $M_n(S')^+$ spans $M_n(S')$. 
We remark that  if $S^d$ turns out to be an operator system, then $S^d = S'$. 
\end{remark}

There exists infinite-dimensional operator system whose matrix-ordered dual is as well an operator system. For example, Paulsen and the author in \cite{NP2016} constructed the operator Hilbert system $SOH$, whose matrix-ordered dual remains to be an operator system. Below we give another example. 

\begin{exam}
Given an operator space $V$, the \textit{Paulsen system} $S(V)$ is 
	\begin{equation*}
		S(V) := \left\{ \begin{bmatrix} \lambda I & X \\ Y^* & \mu I	\end{bmatrix} \in M_2(B(\mathcal{H})) \colon \lambda, \mu \in \bb{C}, X, Y \in V 			\right\}.
	\end{equation*}
In \cite{Pa2}, it is shown that $S(V)$ is independent of the representation $V\subset B(\mathcal{H})$, up to complete order isomorphism. One can check that the \textit{trace} functional
	\begin{equation*}
		tr \left( \begin{bmatrix} \lambda I & X \\ Y^* & \mu I	\end{bmatrix} \right) := \lambda + \mu
	\end{equation*}
is an Archimedean matrix order unit for $S(V)'$; thus, $S(V)'$ is an operator system. 
\end{exam}

\subsection{Inductive limits}\label{inductlim}
In \cite{MT}, Mawhinney and Todorov constructed inductive limits in \textbf{OU}, \textbf{AOU}, \textbf{MOU}, and \textbf{OS}. In this subsection we summarize their main results on \textbf{AOU} and \textbf{OS}. We start with inductive limit in a general category \textbf{C}. 

\begin{defn}
Let \textbf{C} be a category. An \textbf{inductive sequence} in \textbf{C} is a sequence of pair $(A_k, f_k)_{k \in \mathbb{N}}$, where $A_k$ is an object and $f_k$ is a morphism such that $f_k \colon A_k \to A_{k+1}$, for each $k$. 
To avoid excessive notation, we denote it by $(A_k, f_k)$ whenever the content is clear. We call $f_k$ the connecting morphisms. Observe that for $l > k$, $f_{k, l} := f_{l,l} \circ f_{l-1} \circ \dots \circ f_k$ with $f_{k,k} := id_{A_k}$ is a morphism from $A_k$ to $A_l$. 

A pair $(A, \{g_k \}_{k\in \mathbb{N}})$, where $A$ is an object in \textbf{C} and for each $k \in \mathbb{N}$, $g_k \colon A_k \to A$ is a morphism, is said to be \textbf{compatible} with $(A_k, f_k)$, provided for each $k \in \mathbb{N}$, $g_{k+1} \circ f_k  = g_k$.
An \textbf{inductive limit} of $(A_k, f_k)$ is a compatible pair $(A_{\infty}, \{f_{k, \infty}\}_{k\in\mathbb{N}})$ that satisfies the universal property: If $(B, \{g_k\}_{k \in \mathbb{N}})$ is another compatible pair with $(A_k, f_k)$, then there exists a unique morphism $u \colon A_{\infty} \to B$ such that $u \circ f_{k,\infty} = g_k $, for each $k$. 

If $(A_k, f_k)$ has an inductive limit, then it is unique up to isomorphism in \textbf{C}, and it will be denoted  $(A_{\infty}, \{f_{k,\infty}\}_{k\in\mathbb{N}} )$ or $(A_{\infty}, \{f_{k,\infty}\})  = \inductlim{C}(A_k, f_k)$, or $A_{\infty} = \inductlim{C}A_k$ whenever the content is clear. 
\end{defn}

\subsubsection{OU and AOU spaces}\label{inductiveOUandAOU}
We omit the details in \cite[\S 3]{MT} but outline some basic facts briefly. Given an inductive sequence  $( (V_k, V_k^+, e_k), \phi_k)$ in \textbf{OU}, we consider the vector subspace $V_{\infty}^0$ in $\prod_{k\in\mathbb{N}} V_k$, where
	\begin{equation*}
		V_{\infty}^0 = \{ (x_k) \in \prod_{k\in\mathbb{N}}(V_k) \colon \phi_k(x_k) = x_{k+1}, \forall k \in \mathbb{N} \}.	
	\end{equation*}
Let $N^0 = \{ (x_k) \in V_{\infty}^0 \colon \exists m \in \mathbb{N} \; \text{so that} \; \forall k \geq m, x_k = 0 \}$ and $\ddot{V}_{\infty} := V_{\infty}^0 / N^0$. There are unital positive maps $\ddot{\phi}_{k,\infty} \colon V_k \to \ddot{V}_{\infty}$ that satisfy the following:
	\begin{enumerate}[label={\upshape(\roman*)}, align=left, widest=iii, leftmargin=*]
		\item 	$\ddot{V}_{\infty} = \cup_{k\in\mathbb{N}} \ddot{\phi}_{k,\infty} (V_k)$, and $\ddot{\phi}_{k,\infty}(x_k) = \ddot{ \phi}_{l,\infty}(x_l)$ if and only if $\phi_{k,m}(x_k) = \phi_{l,m}(x_l)$ for some $m > k,l$. 
		\item 	$\ddot{V}_{\infty}^+$ is the set of $\ddot{\phi}_{k,\infty}(x_k)$ such that there exists $m \geq k$ with $\phi_{k,m}(V_m^+)$. Moreover, $\ddot{V}_{\infty}^+ = \cup_{k\in\mathbb{N}} \ddot{\phi}_{k,\infty} (V_k^+)$. 
		\item 	$(\ddot{V}_{\infty}, \ddot{V}_{\infty}^+)$ together with $\ddot{e} = \ddot{\phi}_{k,\infty}(e_k)$ is an order unit space.
		\item 	$( \ddot{V}_{\infty}, \{\ddot{\phi}_{k,\infty} \} ) = \inductlim{OU}(V_k, \phi_k)$. 
	\end{enumerate}

Given an inductive sequence $(V_k, \phi_k)$ in \textbf{AOU}, we first obtain $\ddot{V}_{\infty}$ in \textbf{OU} and then Archimedeanize $\ddot{V}_{\infty}$ as follows. Let $x = \ddot{\phi}_{k,\infty}(x_k) \in \ddot{V}_{\infty}$ and let $N$ be the \textit{null space} 
	\begin{equation}
		N := \{ x \in \ddot{V}_{\infty} \colon \lim_{m\to\infty} || \phi_{k,m} (x_k) ||^m = 0 \},
	\end{equation} 
where $|| \cdot ||^m$ is any order norm on $V_m$. It follows that $N$ is the kernel of an, hence any, order seminorm $|| \cdot ||^{\infty}$ on $\ddot{V}_{\infty}$. 
Let $V_{\infty} := \ddot{V}_{\infty}/N$, then $V_{\infty}$ is the Archimedeanization of $\ddot{V}_{\infty}$. 

\begin{remark}\label{inductiveAOU}
Let $q_V \colon \ddot{V}_{\infty} \to V_{\infty}$ be the canonical quotient map and $\phi_{k,\infty} = q_V \circ \ddot{\phi}_{k,\infty}$. Then the pair $(V_{\infty}, \{\phi_{k,\infty} \})$ satisfies the following:  
	\begin{enumerate}[label={\upshape(\roman*)}, align=left, widest=iii, leftmargin=*]
		\item $V_{\infty} = \cup_{k\in\mathbb{N}}\phi_{k,\infty} (V_k)$, and $\phi_{k,\infty}(x_k) = \phi_{l,\infty}(x_l)$ if and only if $\ddot{\phi}_{k,\infty}(x_k) - \ddot{\phi}_{l,\infty}(x_l) \in N$, if and only if $\lim_{m\to \infty} || \phi_{k,m} (x_k) - \phi_{l,m}(x_l) ||^m = 0$. 
		\item $V_{\infty}^+$ is the set of $\phi_{k,\infty} (x_k)$, where for every $r > 0$, there exist $m \geq l > k$ such that $y_l \in V_l$ with $\ddot{\phi}_{l,\infty}(y_l) \in N$, and $re_m + \phi_{k,m}(x_k) + \phi_{l,m}(y_l) \in V_m^+$. 
		\item 	$(V_{\infty}, V_{\infty}^+)$ together with $e = {\phi}_{k,\infty}(e_k)$ is an AOU space.
		\item $(V_{\infty}, \{\phi_{k,\infty} \}) = \inductlim{AOU}(V_k, \phi_k)$. 
	\end{enumerate} 
\end{remark}

By dualizing the inductive sequence $(V_k, \phi_k)$ in \textbf{OU}, we obtain the following reverse sequence of ordered duals: 
\begin{equation*}
V_1' \overset{\phi_{k}^{'}}{\longleftarrow} \dots \overset{\phi_{k-1}^{'}}{\longleftarrow} V_k' \overset{\phi_{k}^{'}}{\longleftarrow} V_{k+1}' \overset{\phi_{k+1}^{'}}{\longleftarrow} \dots 
\end{equation*}

Since each $\phi_k$ is unital, $\phi_k'$ maps state space to state space, we have the following reverse sequence of compact Hausdorff topological spaces with respect to weak*-topology:
\begin{equation*}
\mathfrak{S}(V_1) \overset{\phi_{1}^{'}}{\longleftarrow}	\dots \overset{\phi_{k-1}^{'}}{\longleftarrow} \mathfrak{S}(V_k) \overset{\phi_{k}^{'}}{\longleftarrow}	\mathfrak{S}(V_{k+1}) \overset{\phi_{k+1}^{'}}{\longleftarrow} \dots 
\end{equation*}
This is a \textit{projective sequence} in \textbf{TOP} whose objects are topological spaces and morphisms are continuous maps. By \cite{bourbakiTOP}, its projective limit is 
\begin{equation*}
\projectlim{TOP} \mathfrak{S}(V_k) = \left\{ (f_k) \in \prod_{k\in \mathbb{N}} \mathfrak{S}(V_k) \colon f_k = \phi_{k+1}' (f_{k+1}) \right\}, 
\end{equation*}
together with the product topology. By the Tychonoff's theorem, it is a compact and Hausdorff topological space. 

Moreover, there is a homeomorphism $\theta \colon \mathfrak{S}(\ddot{V}_{\infty}) \to \projectlim{TOP} \mathfrak{S}(V_k)$ by $\theta(f) = (f_k)$, where $f_k \in \mathfrak{S}(V_k)$ is defined to be $f_k(v_k) := f( \ddot{\phi}_{k,\infty} (v_k) )$. When $(V_k, \phi_k)$ is in \textbf{AOU}, $\mathfrak{S}(V_{\infty})$ is also homeomorphic to $\projectlim{TOP} \mathfrak{S}(V_k)$ by a standard lifting technique on quotient, see \cite[Proposition 3.16]{MT}. 

\subsubsection{MOU and OS} Again we briefly describe the constructions here.
Let $(S_k, \phi_k)$ be an inductive sequence in \textbf{MOU}. At each matrix level, we have an inductive sequence $(M_n(S_k), \phi_k^{(n)})$ in \textbf{OU}. Hence, for each $n \in \mathbb{N}$, we obtain $\inductlim{OU} M_n(S_k)$. Let $\ddot{S}_{\infty} = \inductlim{OU} S_k$. It turns out that $M_n(\ddot{S}_{\infty})$, equipped with canonical structures, is order isomorphic to $\inductlim{OU} M_n(S_k)$. Moreover, the maps $\ddot{\phi}_{k,\infty}$ become unital completely positive; then it follows that the pair $(\ddot{S}_{\infty}, \{\ddot{\phi}_{k,\infty} \}) = \inductlim{MOU}(S_k,\phi_k)$.

When $(S_k,\phi_k)$ is in \textbf{OS}, we archimedeanize $\inductlim{MOU}(S_k,\phi_k)$ to obtain an operator system $S_{\infty}$. It is shown that $S_{\infty}$ with unital completely positive maps $\phi_{k,\infty} := q_S \circ \ddot{\phi}_{k,\infty}$ is the inductive limit of $(S_k, \phi_k)$ in \textbf{OS}. 

\begin{remark}
Let $(S_k, \phi_k)$ be an inductive sequence in \textbf{OS}. Then for each $n \in \mathbb{N}$, $(M_n(S_k), \phi_k^{(n)})$ is an inductive sequence in \textbf{AOU}. An alternative way to construct $\inductlim{OS}S_k$ is to first obtain $S_{\infty} = \inductlim{AOU} S_k$, and then identify $M_n(S_{\infty})$ to $\inductlim{AOU} M_n(S_k)$ for each $n$. This argument is valid due to the fact that Archimedeanization from \textbf{MOU} to \textbf{OS} is obtained by forming Archimedeanization from \textbf{OU} to \textbf{AOU} at each matrix level.
\end{remark}

\section{Projective limits}\label{Section:Projective limits}
In this section, we construct projective limits in \textbf{AOU} and \textbf{OS}. We start with the definition of projective limit in a general category \textbf{C}. 

\begin{defn}
Let \textbf{C} be a category. A \textbf{projective sequence} in \textbf{C} is a sequence of pairs $\{ (A_k, f_k) \}_{k \in \mathbb{N}}$, where $A_k$ is an object and $f_k$ is a morphism such that $f_k \colon A_{k+1} \to A_k$, for each $k$.
To avoid excessive notation, we denote it by $(A_k, f_k)$ whenever the content is clear. We call $f_k$ the connecting morphisms. Observe that for $l < k$, $f_{k, l} := f_{l} \circ \dots \circ f_{k-1} \circ f_{k,k}$ with $f_{k,k} := id_{A_k}$ is a morphism from $A_k$ to $A_l$.  

A pair $(A, \{p_k \}_{k\in \mathbb{N}})$, where $A$ is an object in \textbf{C} and for each $k \in \mathbb{N}$, $p_k \colon A \to A_k$ is a morphism, is said to be \textbf{compatible} with $(A_k, f_k)$, provided for each $k \in \mathbb{N}$, $f_k \circ p_{k+1} = p_k$. A \textbf{projective limit} of $(A_k, f_k)$ is a compatible pair $(A, \{p_k\}_{k\in\mathbb{N}})$ that satisfies the universal property: If $(B, \{q_k\}_{k \in \mathbb{N}})$ is another compatible pair with $(A_k, f_k)$, then there exists a unique morphism $u \colon B \to A$ such that $p_k \circ u = q_k $, for each $k$. 
\end{defn}

\begin{remark}
If $(A_k, f_k)$ has a projective limit $(A, \{ p_{k}\}_{k\in\mathbb{N} }  )$, then it is unique up to isomorphism in \textbf{C}, and it will be denoted  $(A, \{p_k\}_{k\in\mathbb{N} })$ or $(A, \{p_k\}) = \projectlim{C}(A_k, f_k)$, or $A=\projectlim{C}A_k$ whenever the content is clear. We summarize the above in the following diagram.
\begin{equation}\label{Cuniversal_property}
\xymatrix{
A_1	&\ar[l]_{f_1}	\dots	&\ar[l]_{f_{k-1}} A_k	&\ar[l]_{f_k} A_{k+1}	& \ar[l]_{f_{k+1}} \dots \\
&	&	A	\ar[u]^{p_k} \ar@/_0.5pc/[ur]^{p_{k+1}} \ar@/^0.5pc/@{=>}[ul] \ar@/^1.0pc/[ull]^{p_1} \ar@/_1.0pc/@{=>}[urr]&	& 
}
\end{equation} 
\end{remark}

\subsection{Projective limits of TVS}\label{projectlimTVS}
Our ambient category is topological vector spaces with continuous linear maps, denoted by \textbf{TVS}. 
The material in this subsection is standard, see \cite{bourbakiTOP}. We recall a few facts that shall be used in the sequel. 

Given a family $\{V_i\}_{i \in I}$ of topological vector spaces, the Cartesian product $X := \prod_{i \in I} V_i = \{v=(v_i) \colon v_i \in V_i, i \in I \}$ endowed with the product topology, is a topological space. For each $i \in I$, the \textit{projection} $\pi_i \colon X \to V_i$ by $\pi_i (v) := v_i$ is a continuous map. 
Also, $X$ equipped with the canonical vector addition and scalar multiplication is a topological vector space, hence is a \textbf{TVS}. If each $V_i$ is Hausdorff, then so is $X$ by the Tychonoff's theorem. 
Henceforth, we take $I = \mathbb{N}$ with the usual partial ordering. 

\begin{thm}\label{TVSprojective_limit}
Suppose $(V_k, \phi_k)_{k \in \mathbb{N}}$ is a projective sequence in \textbf{TVS}. Define 
\begin{equation*}
V := \left\{ v = (v_k) \in \prod_{k \in \mathbb{N}} V_k \colon \phi_k(v_{k+1}) = v_k, \; \text{for all} \; k \in \mathbb{N} \right\},
\end{equation*} 
and define the map $p_k \colon V \to V_k$ by $p_k := \pi_k \circ \iota$ for each $k \in \mathbb{N}$, where $\iota$ is the inclusion map from $V$ to $X$. Then $(V, \{ p_k \} )$ is the projective limit of $(V_k, \phi_k)$ in \textbf{TVS}. 
\end{thm}

The existence part is standard; the uniqueness part will be emphasized in the following remark and proposition.

\begin{remark}
Indeed, $V$ is a closed subspace of $X$ with respect to the product topology, which is equivalent to the initial topology induced by $\{\pi_k\}_{k \in \mathbb{N}}$. The subspace topology on $V$ is equivalent to the \textit{initial topology} induced by $\{p_k\}_{k \in \mathbb{N}}$. Hence, if we first construct the projective limit $(W, \{q_k\}_{k \in \mathbb{N}})$ of $(V_k, \phi_k)$ in \textbf{VS}, and then equip $W$ with the initial topology induced by $\{q_k\}_{k \in \mathbb{N}}$, then $V$ and $W$ are isomorphic in \textbf{TVS}.  This observation and the next proposition conclude the uniqueness of projective limit in \textbf{TVS}. 
\end{remark}

\begin{prop}\label{uniquemorphism}
Let $(V, \{p_k\})$ be given as in Theorem \ref{TVSprojective_limit} and $W$ be a topological vector space. Then a liner map $\psi \colon W \to V$ is continuous if and only if $p_k \circ \psi \colon W \to V_k$ is continuous for each $k \in \mathbb{N}$. Moreover, if $(W, \{q_k\})$ in \textbf{TVS} is compatible with $(V_k, \phi_k)$, then there exists unique continuous linear map $\psi \colon W \to V$ such that $p_k \circ \psi= q_k$, for each $k \in \mathbb{N}$.
\end{prop}

\begin{proof}
The first statement is a direct consequence of characterization of initial topology. For the second statement, note that the morphisms $q_k$ determine a unique morphism $\psi \colon W \to X$ by $\psi(w) := (q_k(w))$. The compatibility condition asserts that the image of $\psi$ is $V$ and $p_k \circ \psi = q_k$.  
\end{proof}

\begin{remark}
It follows that the uniqueness of projective limit in \textbf{AOU} and \textbf{OS} are deducible from this proposition. It is also the primary reason the author considered \textbf{TVS} instead of \textbf{VS} and emphasized the weak*-topology on matrix-ordered dual in Remark \ref{Remark:ordered_dual}.
\end{remark}

\subsection{Projective limits of AOU}\label{projectlimAOU}
In this subsection, let $(V_k, V_k^+, e_k )_{k \in \mathbb{N}}$ be a sequence of AOU spaces and $\phi_k \colon V_{k+1}$ $\to V_k$ be a unital positive map for each $k \in \mathbb{N}$. We simply write $(V_k, \phi_k)$ for such projective sequence in \textbf{AOU}. 
To construct a candidate for the projective limit of $(V_k, \phi_k)$ in \textbf{AOU}, we begin by working in the ambient space \textbf{TVS}. By the discussion in \S \ref{AOUspaces}, \textbf{AOU} is a subcategory of \textbf{TVS}.

Let us apply the forgetful functor from \textbf{AOU} to \textbf{TVS} and obtain the projective limit $(V, \{ p_k \}) =$ $\projectlim{TVS} (V_k, \phi_k)$ as in Theorem \ref{TVSprojective_limit}. Equip $V$ with canonical involution $v^* := (v_k^*)$, and define $V_h := \{ v \in V \colon v = v^*\}$ and $V^+ := \prod_{k \in \mathbb{N}} V_k^+$. Then $(V, V^+)$ is an ordered $\ast$-vector space such that diagram \eqref{Cuniversal_property} commutes in \textbf{TVS}.

We now proceed to construct a candidate $V_{\infty} \subset V$ in \textbf{AOU}. 
For each $k \in \mathbb{N}$, we denote the minimal and maximal order norms for $v_k \in V_k$ by $|| v_k ||^k_{m}$ and $|| v_k ||^k_{M}$, respectively. Recall that all order norms on $V_k$ are bounded between $|| \cdot ||^k_{m}$ and $|| \cdot ||^k_{M} $, and they all coincide on $(V_k)_h$. For Hermitian $v_k \in (V_k)_h$, we simply write $||v_k||^k_h$ for the order seminorm determined by $e_k$.  
We define for $v = (v_k) \in V$, 
	\begin{equation*}
		|| v ||_{\max}	:= 	\sup_k \{ || v_k ||^k_{M} \} 	\qquad \text{and} 	\qquad	|| v ||_{\min}	:=		\sup_k \{ || v_k ||^k_{m} \}.
	\end{equation*}
If $v \in V_h$, then $||v||_{\min} = ||v||_{\max}$, and we will write $||v ||_h$ to avoid excessive notation. Define
\begin{equation}\label{Vinfty}
V_{\infty} := \{ v \in V \colon || v ||_{\min} < \infty \}.
\end{equation}
Let $(V_{\infty})_h := V_{\infty} \cap V_h$ and $V_{\infty}^+ := V_{\infty} \cap V^+$. It follows that $(V_{\infty}, V_{\infty}^+)$ is an ordered $\ast$-subspace of $V$. We first claim that $(V_{\infty}, V_{\infty}^+)$ with $e := (e_k) \in V_{\infty}^+$ is indeed an AOU space.

\begin{lemma}\label{orderunit}
The triple $(V_{\infty}, V_{\infty}^+, e)$ is an AOU space.
\end{lemma}

\begin{proof}
Given $v = (v_k) \in (V_{\infty})_h$, let $r = || v ||_h$, then for each $k \in \mathbb{N}$, $re_k + v_k \geq || v_k ||^k_h e_k + v_k \in V_k^+$. Thus $re+v \in V_{\infty}^+$, and $e$ is an order unit for $(V_{\infty}, V_{\infty}^+)$.
To see that $e$ is Archimedean, suppose $v = (v_k) \in (V_{\infty})_h$ and $r e + v \in V_{\infty}^+$ for each $r > 0$. For each $k \in \mathbb{N}$, $r e_k + v_k \in V_k^+$ asserts that $v_k \in V_k^+$. Hence, $v \in V_{\infty} \cap V^+ = V_{\infty}^+$, and $(V_{\infty}, V_{\infty}^+, e)$ is an AOU space.
\end{proof}

\begin{remark}
Since each $|| \cdot ||^k_{m}$ (resp. $|| \cdot ||^k_{M}$)  is a $\ast$-norm on $V_k$, it is evident that $|| \cdot ||_{\min}$ (resp. $|| \cdot ||_{\max}$) is a $\ast$-norm on $V_{\infty}$. In the next lemma, we will see that $|| \cdot ||_h$ is indeed the order seminorm on $(V_{\infty})_h$ determined by $e$, and thus $|| \cdot ||_{\min}$ and $|| \cdot ||_{\max}$ are order norms on $V_{\infty}$.
We also remark that we can replace $|| \cdot ||_{\min}$ with $|| \cdot ||_{\max}$ in \eqref{Vinfty} since all order norms are equivalent on each $V_k$. 
\end{remark}

\begin{lemma}\label{order_seminorm}
The seminorm $|| \cdot ||_h$ defined on $(V_{\infty})_h$ is the order seminorm determined by $e$. Moreover, $|| \cdot ||_{\min}$ and $|| \cdot ||_{\max}$ are order norms on $(V_{\infty}, V_{\infty}^+, e)$. 
\end{lemma}

\begin{proof}
Note that $e$ is an order unit for $(V_{\infty})_h$. For $v \in (V_{\infty})_h$, denote $|| v ||$ the order seminorm determined by $e$ as in \eqref{OrderNorm}. We shall show that $|| v || = || v ||_h$, for $v \in (V_{\infty})_h$. Indeed, if $r = ||v||$, then $re \pm v \in V_{\infty}^+$, and for each $k \in \mathbb{N}$, $re_k \pm v_k \in V_k^+$, which implies that $r \geq || v_k ||^k_h$. Hence, $|| v ||_h = \sup_k ||v_k||^k_h \leq r = || v ||$. 

Conversely, let $||v|| > \varepsilon > 0$ and $r_{\varepsilon} = || v || - \varepsilon$. By definition of order seminorm, $r_{\varepsilon} e + v$ or $r_{\varepsilon} e - v$ is not positive. It follows that there exists $k \in \mathbb{N}$ such that $r_{\varepsilon} e_k + v_k$ or $r_{\varepsilon} e_k - v_k$ is not in $V_k^+$. Since $||v_k||^k_h e_k \pm v_k \in V_k^+$, again by definition, $r_{\varepsilon} < ||v_k||^k_h \leq || v ||$; therefore, $|| v || = || v ||_h$. 
The last statement follows from $|| v ||_{\min} = || v ||_h = || v ||_{\max}$, for all $v \in (V_{\infty})_h$. 
\end{proof}

\begin{prop}\label{AOUcompatible}
For each $k\in\mathbb{N}$, let $p_{\infty, k} \colon V_{\infty} \to V_k$ by $p_{\infty, k}:= p_k \circ \iota_{\infty}$,  where $\iota_{\infty}$ is the natural inclusion from $V_{\infty}$ to $V$. Then $p_{\infty,k}$ is a unital positive map. Furthermore, the pair $(V_{\infty}, \{p_{\infty, k} \}_{k \in \mathbb{N}} )$ is compatible with $(V_k, \phi_k)$ in \textbf{AOU}.
\end{prop}

\begin{proof}
For each $k \in \mathbb{N}$, $p_{\infty,k}(v) = p_k \circ \iota_{\infty} (v) = v_k$ is a unital positive linear map. Moreover, $\phi_{k} \circ p_{\infty, k+1}(v) = \phi_{k} ( v_{k+1} ) = p_{k}(v) = p_{\infty, k} \circ\iota_{\infty} (v) = p_{\infty, k} (v)$. Therefore, $(V_{\infty}, \{ p_{\infty, k}\} )$ is compatible with $(V_k, \phi_k)$ in \textbf{AOU}. 
\end{proof}

\begin{remark}
We summarize the relations above in the diagram below.

\begin{equation}
\xymatrix{
V_1	&\ar[l]_{\phi_1}	\dots	&\ar[l]_{\phi_{k-1}} V_k	&  \dots \ar[l]_{\phi_{k}}  \\
&	&	V_{\infty}	\ar[u]_{p_{\infty,k}}  \ar@{=>}[ul] 	\ar@/^0.5pc/[ull]^{p_{\infty,1}} \ar@/_0.5pc/@{=>}[ur] \ar[d]^{\iota_{\infty}}&	& 	\\
&	&	V	\ar@/^3.0pc/[lluu]^{p_1} \ar@/^4.0pc/@{=>}[luu]	\ar@/^2.0pc/[uu]^{p_k}	\ar@/_4.0pc/@{=>}[ruu] & 
}
\end{equation} 

We remark that, if $(V_k, \phi_k)$ are in \textbf{OU}, then the order norms used above will be order seminorms. The same construction will still go through; in particular, similar arguments in Lemma \ref{orderunit}	and Proposition \ref{AOUcompatible} will show that $(V_{\infty}, \{ p_{\infty, k }\} )$ is a compatible pair with $(V_k, \phi_k)$ in \textbf{OU}. However, what makes the construction rather complicated is the topological structures and continuity because $V_k$ and $V_{\infty}$ need not be Hausdorff. 
\end{remark}

\begin{thm}\label{AOUprojective_limit}
The pair $( (V_{\infty}, V_{\infty}^+, e), \{ p_{\infty, k} \}_{k\in\mathbb{N}})$ is the limit of the projective sequence $( (V_k, V_k^+, e_k), \phi_k )$ in \textbf{AOU}.
\end{thm}

\begin{proof}
Suppose $( (W, W^+, e') , \{q_k\}_{k\in\mathbb{N}})$ in \textbf{AOU} is another compatible pair with $( (V_k, V_k^+, e_k), \phi_k)$. We will show that there exists unique unital positive map $\psi \colon W \to  V_{\infty}$ such that for each $k \in \mathbb{N}$, $p_{\infty, k} \circ \psi = q_k$.

Note that $(W, \{q_k\})$ is compatible with $(V_k, \phi_k)$ in \textbf{TVS}, so by Proposition \ref{uniquemorphism}, there exists unique continuous linear $\psi \colon W \to V$ such that $p_{k} \circ \psi = q_k$. We will show that $\psi$ is indeed a unital positive map from $W$ into $V_{\infty}$. By Proposition \ref{upContractive}, since $q_k$ is unital positive, for each $w \in W$, $|| q_k (w) ||^k_{m} \leq || w ||_{m}$ for each $k \in \mathbb{N}$; so $|| \psi(w) ||_{\min} < \infty$ and $\psi(w) \in V_{\infty}$. Also, $\psi(e') = ( q_k(e') ) = (e_k) = e$; and for $w \in W^+$, $\psi(w) = (q_k(w)) \in V_{\infty}^+$. 

The uniqueness of $\psi$ in \textbf{AOU} follows from its uniquenss in \textbf{TVS}. Consequently, $\psi \colon W \to V_{\infty}$ is the unique morphism in \textbf{AOU} that commutes in the diagram below.  
	\begin{equation}\label{AOUuniversal_diagram}
		\xymatrix{
				V_k 	& & V	\\
				W	\ar@/^1.75pc/@{-->}[urr]^{\psi} \ar[u]^{q_k}	\ar@{-->}[rr]_{\psi} & & V_{\infty} \ar[u]_{\iota_{\infty}} \ar[llu]_{p_{\infty, k}} 	
		}			
	\end{equation}
\end{proof}

We end this subsection by proving that $|| \cdot ||_{\min}$ and $|| \cdot ||_{\max}$ are indeed the minimal and maximal order norms, respectively, on $V_{\infty}$. 

\begin{prop}\label{AOUminOrderNorm}
On $(V_{\infty}, V_{\infty}^+, e)$, $|| \cdot ||_{\min}$ is the minimal order norm. 
\end{prop}

\begin{proof}
Denote $|| \cdot ||_m$ the minimal order norm on $V_{\infty}$. By Lemma \ref{order_seminorm}, $|| \cdot ||_{\min}$ is an order norm, so $|| \cdot ||_m \leq || \cdot ||_{\min}$. For the converse, note that for each $f_k \in \mathfrak{S}(V_k)$, $f_k \circ p_{\infty, k} \in \mathfrak{S}(V_{\infty})$. Hence for $v = (v_k) \in V_{\infty}$,
	\begin{align*}
		|| v ||_{\min} &= \sup_k || v_k ||^k_{\min} = \sup_k  \left\{	\sup \{ | f_k (v_k) | \colon f_k \in \mathfrak{S}(V_k)		\} \right\}	\\
							&=	\sup \left\{ | f_k \circ p_{\infty, k} (v) | \colon k \in \mathbb{N}, f_k \in \mathfrak{S}(V_k)		\right\}		\\
							&\leq \sup \{ |f (v) | \colon f \in \mathfrak{S}(V_{\infty})	\} = || v ||_{m}. 
	\end{align*}
Therefore, $|| v ||_{\min} = ||v||_m$. 
\end{proof}

\begin{prop}
On $(V_{\infty}, V_{\infty}^+, e)$, $|| \cdot ||_{\max}$ is the maximal order norm. 
\end{prop}

\begin{proof}
Denote $|| \cdot ||_M$ the maximal order norm on $V_{\infty}$. By Lemma \ref{order_seminorm}, $|| \cdot ||_{\max}$ is an order norm, so $|| \cdot ||_{\max} \leq || \cdot ||_{M}$. For the converse, consider $v = (v_k) \in V_{\infty}$. Fix a representation $v = \sum_{i=1}^n \lambda_i v_i$, where $\lambda_i \in \mathbb{C}$ and $v_i = (v_i^k)_{k \in \mathbb{N}} \in (V_{\infty})_h$, so that
	\begin{equation*}
		\sum_{i=1}^n | \lambda_i| || v_i ||_h = \sum_{i=1}^n |\lambda_i| \sup_k || v_i^k ||^k_h	=	\sup_k \sum_{i=1}^n |\lambda_i| || v_i^k ||^k_h.
	\end{equation*}
For each $k \in \mathbb{N}$, this representation of $v$ implies that $v_k = \sum_{i=1}^n \lambda_i v_i^k$, with $v_i^k \in (V_k)_h$. In particular, by definition of maximal order norm, we have $\sum_{i=1}^n |\lambda_i| || v_i^k ||^k_h \geq ||v_k||_{\max}^k$. Hence, 
	\begin{equation*}
		\sum_{i=1}^n | \lambda_i| || v_i ||_h \geq \sup_k ||v_k||^k_{\max} = || v ||_M.
	\end{equation*}
Taking the infimum over all such representations yields that $|| v ||_{\max} \geq || v ||_M$. Therefore, $||v||_{\max} = ||v||_M$. 
\end{proof}

\subsection{Projective limits of OS}\label{projectlimOS}
We proceed to construct projective limit in \textbf{OS}. A linear map $\phi$ between operator systems $S$ and $T$ is unital completely positive if and only if for each $n \in \mathbb{N}$, its amplification $\phi^{(n)} = id_n \otimes \phi \colon M_n(S) \to M_n(T)$ is unital positive. In particular, $(M_n(S), M_n(S)^+, I_n \otimes e_S)$ is an AOU space and likewise for $M_n(T)$. Hence, a projective sequence $(S_k, \phi_k)$ in \textbf{OS} gives rise,  for each $n \in \mathbb{N}$, to a  projective sequence $( M_n(S_k), \phi_k^{(n)})_{k \in \mathbb{N}}$ in \textbf{AOU}.
When $n = 1$, by Theorem \ref{AOUprojective_limit}, we denote $(S_{\infty}, S_{\infty}^+, e)$ with morphisms $\{\phi_{k, \infty}\}_{k\in\mathbb{N}}$ the projective limit $\projectlim{AOU} (S_k, \phi_k)$.

The key step of the construction is to realize that there is a natural AOU structure on the vector space $M_n(S_{\infty})$ induced by $\{M_n(S_k)^+\}_{k\in\bb{N}}$, which yields a matrix ordering on $S_{\infty}$. To avoid confusion on reading, in this subsection we denote $\vec{x} = (x_k) \in S_{\infty}$ and $[x^k_{ij}]$ for an element in $M_n(S_k)$. 
There is a canonical vector space identification between $M_n(S_{\infty})$ and $\projectlim{AOU} M_n(S_k)$; namely, given $[\vec{x}_{ij}] \in M_n(S_{\infty})$, we identify it to $( [x^k_{ij}] )_{k\in\mathbb{N} }$, and vice-versa. We shall see that this identification endows $M_n(S_{\infty})$ the desired structure. 

Given $[\vec{x}_{ij}] \in M_n(S_{\infty})$, define $[\vec{x}_{ij}]^* := [ \vec{x}^*_{ji} ] = [ (x^k)^*_{ji} ]$, then $M_n(S_{\infty})$ is a $\ast$-vector space. 
We define a matrix ordering on $S_{\infty}$ by  
\begin{equation}\label{projectiveOScone}
M_n(S_{\infty})^+ := \{ [\vec{x}_{ij}] \in M_n(S_{\infty}) \colon [ x^k_{ij} ] \in M_n(S_k)^+, \forall k \in \mathbb{N} \}.
\end{equation}

Note that for each $[\vec{x}_{ij}] \in M_n(S_{\infty})^+$ and $\alpha \in M_{m,n}(\mathbb{C})$, $\alpha [\vec{x}_{ij}] \alpha^* = ( \alpha [x^k_{ij}] \alpha^* )_{k\in\mathbb{N}}$ is in $M_m(S_{\infty})^+$. Hence, this definition is well-defined by linearity and compatibilty of $\{\phi_{\infty, k}\}_{k\in\mathbb{N}}$. We shall show that it indeed defines a matrix ordering on $S_{\infty}$. 

\begin{prop}
The triple $(M_n(S_{\infty}), M_n(S_{\infty})^+, I_n \otimes e)$ is an AOU space. Moreover, the collection $\{M_n(S_{\infty})^+\}_{n=1}^{\infty}$ is a compatible family on $S_{\infty}$. Consequently, $(S_{\infty}, \{M_n(S_{\infty})^+\}_{n=1}^{\infty}, e)$ is an operator system.
\end{prop}

\begin{proof}
Note that $M_n(S_{\infty})^+$ is a proper cone in $(M_n(S_{\infty}))_h$ since each of $M_n(S_k)^+$ is a proper cone in $(M_n(S_k))_h$. It suffices to show that $I_n \otimes e$ is an Archimedean order unit. 

Given Hermitian $[\vec{x}_{ij}]$, take $r_{ij} = || \vec{x}_{ij} ||_{\max}$ and let $r = n \cdot \max_{ij} r_{ij}$. For each $k$, $[x^k_{ij}]$ is Hermitian in $M_n(S_k)$, so by Lemma \ref{op-norm_estimate}, $|| [x^k_{ij}] ||_h  = || [x^k_{ij}] ||_{op} \leq r$,
and $r I_n \otimes e_k - [x^k_{ij}] \in M_n(S_k)^+$. Hence, $r I_n \otimes e - [\vec{x}_{ij}] \in M_n(S_{\infty})^+$, and $I_n \otimes e$ is an order unit.
It is Archimedean for if $rI_n\otimes e + [\vec{x}_{ij}] \in M_n(S_{\infty})^+$ for each $r > 0$, then for each $k\in \mathbb{N}$, $r I_n \otimes e_k + [x^k_{ij}] \in M_n(S_k)^+$. Hence, it follows that $[x^k_{ij}] \in M_n(S_k)^+$ and $[\vec{x}_{ij}] \in M_n(S_{\infty})^+$. 

Therefore, we conclude that $(M_n(S_{\infty}), M_n(S_{\infty})^+, I_n\otimes e)$ is an AOU space for every $n \in \mathbb{N}$. Compatibility of the cones $\{M_n(S_{\infty})^+\}_{n=1}^{\infty}$ follows from the definition and compatibilty of $\{ M_n(S_k)^+ \}_{n=1}^{\infty}$ for each $k\in\mathbb{N}$. 
\end{proof}

Now the AOU projective limit $S_{\infty}$ has an operator system structure. We claim that it is also the projective limit of $(S_k, \phi_k)$ in  \textbf{OS} with the same maps $\phi_{\infty, k}$.

\begin{thm}\label{OS_projective_limit}
The pair $( (S_{\infty}, \{M_n(S_{\infty})^+\}_{n=1}^{\infty}, e ), \{ \phi_{\infty, k} \}_{k \in \mathbb{N}} )$ is a projective limit of $(S_k, \phi_k)$ in \textbf{OS}.	
\end{thm}

\begin{proof}
We claim that for each $n \in \mathbb{N}$, $(M_n(S_{\infty}), M_n(S_{\infty})^+, I_n \otimes e)$ with the maps $\phi_{\infty,k}^{(n)}$, is a projective limit of $(M_n(S_k), \phi_k^{(n)})_{k\in\mathbb{N}}$ in \textbf{AOU}. First we show that the map $\phi_{\infty,k}^{(n)}$ is positive. Indeed, if $[\vec{x}_{ij}] \in M_n(S_{\infty})^+$, then
	\begin{align*}
		\phi_{\infty,k}^{(n)} ([\vec{x}_{ij}]) &= [ \phi_{\infty,k} (\vec{x}_{ij}) ] = [x^k_{ij} ],
	\end{align*}
is in $M_n(S_k)^+$ by definition of $M_n(S_{\infty})^+$. In particular, we see that $\phi_{\infty,k}^{(n)}$ is unital. Moreover, 
	\begin{align*}
		\phi_k^{(n)} \circ \phi_{\infty, k+1}^{(n)} &= (id_n \otimes \phi_k) \circ (id_n \otimes \phi_{\infty,k+1} ) \\
			&= id_n \otimes ( \phi_k \circ \phi_{\infty, k+1} ) 
			= id_n \otimes \phi_{\infty, k} = \phi_{\infty, k}^{(n)},
	\end{align*}
so $(M_n(S_{\infty}), \{ \phi_{\infty, k}^{(n)} \}_{k\in\mathbb{N}} )$ is a compatible pair with $(M_n(S_k), \phi_k^{(n)})$ in \textbf{AOU}; consequently, $(S_{\infty}, \{ \phi_{\infty, k} \}_{k\in\mathbb{N}} )$ is compatible with $(S_k, \phi_k)$ in \textbf{OS}.

For the universal property, let $(T, \{ \psi_k \}_{k\in\mathbb{N}})$ be compatible with $(S_k, \phi_k)$ in \textbf{OS}.  
Define $\Psi \colon T \to S_{\infty}$ to be $\Psi (t) := ( \psi_k(t) )_{k \in \mathbb{N}}$. Then $\Psi$ is well-defined, unital, and positive just as in the proof of Theorem \ref{AOUprojective_limit}. Moreover, for each $n \in \mathbb{N}$ and every $Y \in M_n(T)^+$, $\Psi^{(n)}( Y )$ can be identified to $( \psi_k^{(n)} (Y) )_{k \in \mathbb{N}}$, where each $\psi_k^{(n)} (Y) \in M_n(S_k)^+$. Thus, $\Psi$ is completely positive, and it is evident that $\Psi$ is unital. 
The uniqueness of $\Psi$ follows from the universal property of $S_{\infty}$ in \textbf{AOU} at the ground level. 
\end{proof}

\begin{remark}\label{surjectivity}
If each $\phi_k$ is surjective, the map $\phi_{\infty, k}$ is also surjective for each $k$. Indeed, given $x_k \in S_k$, take $\vec{x} = (x_k) \in S_{\infty}$ such that $x_i = \phi_{k,i}(x_k)$ for $i \leq k$ and $x_i \in \phi_{i,k}^{-1}( \{x_k \} ) \neq \emptyset$ for $i > k$. Then $\phi_{\infty, k}(\vec{x}) = x_k$. 
\end{remark}

\begin{thm}\label{OSdiagram_commute}
Let $(S_k, \phi_k)$ and $(T_k, \psi_k)$ be two projective sequences in \textbf{OS} such that for each $k$, there is unital completely positive $\theta_k \colon S_k \to T_k$, where $\theta_k \circ \phi_k = \psi_k \circ \theta_{k+1}$ for each $k \in \mathbb{N}$. Suppose $(S_{\infty}, \{\phi_{\infty,k} \})$ and $(T_{\infty}, \{ \psi_{\infty, k} \})$ are the projective limits of $(S_k, \phi_k)$ and $(T_k, \psi_k)$, respectively. Then there exists unique morphism $\Theta \colon S_{\infty} \to T_{\infty}$ such that  $\theta_k \circ \phi_{\infty, k} = \psi_{\infty,k} \circ \Theta$ for every $k \in \mathbb{N}$. 
\end{thm}

\begin{proof}
The pair $(S_{\infty}, \{ \theta_k \circ \phi_{\infty, k} \}_{k \in \mathbb{N}})$ is compatible with $(T_k, \psi_k)$ since $\psi_{k} \circ (\theta_{k+1} \circ \phi_{\infty, k+1} ) = (\theta_k \circ \phi_k) \circ \phi_{\infty, k+1} = \theta_k \circ \phi_{\infty, k}$. By the universal property of $T_{\infty}$, there exists unique morphism $\Theta \colon S_{\infty} \to T_{\infty}$ such that $\psi_{\infty, k} \circ \Theta = \theta_k \circ \phi_{\infty,k}$.
\end{proof}

\section{Duality with injective limits}\label{Section:Duality injective}
Let $(S_k, \phi_k)$ be an inductive sequence in \textbf{AOU} (resp. \textbf{OS}), where each $\phi_k$ is an (resp. complete) order embedding and each $S_k'$ is an AOU space (resp.  operator system) with some suitable choice of Archimedean (resp. matrix) order unit.  
Under these assumptions, we show that injective and projective limits of the corresponding sequences are in duality. 

\subsection{AOU Spaces}\label{dualityAOU}
Suppose $(S_k, \phi_k)$ is an inductive sequence in \textbf{AOU} such that each $\phi_k$ is an order embedding. Furthermore, suppose for each $k \in \mathbb{N}$, $(S_k', (S_k')^+, \delta_k)$ is an AOU space such that $\phi_k' \colon S_{k+1}' \to S_k'$ is unital. It follows that $\phi_k'$ is surjective unital positive. Hence, the inductive sequence $(S_k, \phi_k)$ in \textbf{AOU} induces a projective sequence of ordered dual spaces $(S_k', \phi_k')$ in \textbf{AOU}. 
By \S \ref{inductiveOUandAOU} or \cite[\S 3]{MT}, let $(S_{\infty}, \{ \phi_{k,\infty} \}_{k\in\mathbb{N}} )$ be the inductive limit of $(S_k, \phi_k)$. Let $(T_{\infty}, \{ \psi_{\infty,k} \}_{k\in\mathbb{N}} )$ be the projective limit of $(S_k', \phi_k')$ in \S \ref{projectlimAOU}.  By Remark \ref{surjectivity}, the map $\psi_{\infty, k} \colon T_{\infty} \to S_k'$ is surjective.

\begin{remark}
We caution the reader that we will unambiguously use notation $\phi_{k,l}$ for the connecting morphism from $S_k$ to $S_l$, for $k \leq l$, as defined in \S \ref{inductlim}. We also write $\phi_{k,l}'$ for the dual map of $\phi_{k,l}$, so that $\{\phi_{k,l}'\}$ are the connecting morphisms for the projective sequence $(S_k', \phi_k')$. We summarize these assumptions as follows. 

\begin{equation*}
\xymatrix{
	S_k	\ar[r]^{\phi_{k,l}} \ar[rd]_{\phi_{k,\infty}}	&		S_{l} \ar[d]^{\phi_{l,\infty}}		&		&  S_k'	&	S_{l}'	\ar[l]_{\phi_{k,l}'} \\
		&		S_{\infty} 	&	&	T_{\infty} \ar[u]^{\psi_{\infty,k}} \ar[ur]_{\psi_{\infty,l}}	& 
}
\end{equation*}
\end{remark}

\begin{defn}
Define a bilinear pairing $\inner{\cdot}{\cdot} \colon S_{\infty} \times T_{\infty} \to \mathbb{C}$ by
	\begin{equation}\label{duality}
		\inner{\phi_{k,\infty}(x_k)}{(f_m)} := \lim_{m\to\infty} f_m( \phi_{k,m} (x_k ) ).
	\end{equation}
\end{defn}

\begin{prop}\label{AOU:duality}
$S_{\infty}$ and $T_{\infty}$ are in duality induced by \eqref{duality}. 
\end{prop}

\begin{proof}
We first prove that \eqref{duality} is a well-defined bilinear mapping. Let $\phi_{k,\infty}(x_k) \in S_{\infty}$ and $(f_m) \in T_{\infty}$. Let $m \geq k$ and consider
	\begin{align*}\tag{$\dagger$}
		f_m(\phi_{k,m}(x_k))  &= (f_m \circ \phi_{k,m})( \phi_{k,k}( x_k) ) = f_k( x_k ),
	\end{align*}
which shows the limit in \eqref{duality} is in fact a constant. 
To see that it is well-defined, let $\phi_{k,\infty}(x_k) = \phi_{l,\infty}(x_l) \in S_{\infty}$ for some $l \in \mathbb{N}$. By Remark \ref{inductiveAOU}, $|| \phi_{k,m}(x_k) - \phi_{l,m}(x_l) ||^m \to 0$ as $m \to \infty$. Note that $|| f_m || \leq || (f_m) ||_{\max}$ for all $m \in \mathbb{N}$, so
	\begin{equation*}
		|f_m(\phi_{k,m} (x_k) - \phi_{l,m}(x_l) )| \leq  || (f_m) ||_{\max} \cdot || \phi_{k,m} (x_k) - \phi_{l,m} (x_l) ||^m \to 0,
	\end{equation*}
as $m \to 0$, so \eqref{duality} is well-defined. It is easy to see that \eqref{duality} is bilinear. 

For duality, suppose $\inner{\phi_{k,\infty}(x_k)}{(f_m)} = 0$ for every $k \in \mathbb{N}$ and every $x_k \in S_k$. We claim that $(f_m) = 0 \in T_{\infty}$. For fixed $k \in \mathbb{N}$, ($\dagger$) shows that $0 = \inner{\phi_{k,\infty}(x_k)}{(f_m)} = f_k ( \phi_{k,k} (x_k) ) = f_k(x_k)$. By duality between $S_k$ and $S_k'$, we have $f_k = 0$; therefore, $(f_m) = 0 \in T_{\infty}$.

Now suppose $\inner{\phi_{k,\infty}(x_k)}{(f_m)} = 0$ for every $(f_m) \in T_{\infty}$. We must show that $\phi_{k,\infty}(x_k) = 0 \in S_{\infty}$ or $\ddot{\phi}_{k,\infty}(x_k) \in N$; or equivalently, by Remark \ref{inductiveAOU}, $\lim_{l \to \infty} || \phi_{k,l}(x_k) ||^l = 0$. For each $l \geq k$, $| f_l( \phi_{k,l} (x_k) ) | 	= | 	\inner{\phi_{l,\infty}(x_l)}{(f_m)}	| = 0$. By taking the supremum over all $f_l \in S_l'$, we conclude that $|| \phi_{k,l}(x_k) ||^l = 0$; hence, $\phi_{k,\infty}(x_k) = 0 \in S_{\infty}$. 
\end{proof}

\begin{prop}\label{AOUdualcones}
$S_{\infty}^+$ and $T_{\infty}^+$ are dual cones with respect to the dual pair $\inner{S_{\infty}}{T_{\infty}}$. More precisely, 
	\begin{enumerate}[label={\upshape(\roman*)}, align=left, widest=ii, leftmargin=*]
		\item $\phi_{k,\infty}(x_k) \in S_{\infty}^+$ if and only if $\inner{\phi_{k,\infty}(x_k)}{(f_m)} \geq 0$ for all $(f_m) \in T_{\infty}^+$.
		\item $(f_m) \in T_{\infty}^+$ if and only if $\inner{\phi_{k,\infty}(x_k)}{(f_m)} \geq 0$ for all $\phi_{k,\infty}(x_k) \in S_{\infty}^+$. 
	\end{enumerate}
\end{prop}

\begin{proof}
For (i), let $\phi_{k,\infty}(x_k) \in S_{\infty}^+$. By Remark \ref{inductiveAOU}, for each $r > 0$, there exist $m \geq l > k$ and $y_l \in S_l$ with $|| \phi_{l,m}(y_l) ||^m \to 0$, such that $re_m + \phi_{k,m}(x_k) + \phi_{l,m}(y_l) \in S_m^+$. If $(f_m) \in T_{\infty}^+$, then $f_m \in (S_m^d)^+$ and 
	\begin{equation*}
		r f_m(e_m) + f_m( \phi_{k,m}(x_k) ) + f_m( \phi_{l,m} (y_l) ) \geq 0.
	\end{equation*}
Since $|| \phi_{l,m}(y_l)||^m \to 0$ and $|| f_m || \leq || (f_m) ||_{\max}$, letting $m \to \infty$ asserts that $r || (f_m) ||_{\max} + \inner{\phi_{k,\infty}(x_k) }{(f_m)} \geq 0$ for each $r \geq 0$. Consequently, $\inner{\phi_{k,\infty} (x_k) }{(f_m)} \geq 0$. 

Conversely, suppose $\inner{\phi_{k,\infty}(x_k)}{(f_m)} \geq 0$ for all $(f_m) \in T_{\infty}^+$. Then by ($\dagger$), it follows that $f_k(x_k) \geq 0$. Since $\psi_{\infty, k}$ is surjective, letting $f_k$ vary over $(S_k')^+$ asserts that $x_k \in S_k^+$ via duality between $S_k$ and $S_k'$; therefore, $\phi_{k,\infty}(x_k) \in S_{\infty}^+$. 

For (ii), one direction follows just as the first paragraph. Suppose $(f_m) \in T_{\infty}$ and $\inner{\phi_{k,\infty}(x_k)}{(f_m)} \geq 0$ for all $\phi_{k,\infty}(x_k) \in S_{\infty}^+$. Again by ($\dagger$), for each $k \in \mathbb{N}$, $f_k(x_k) = \inner{\phi_{k,\infty}(x_k)}{(f_m)} \geq 0$. Now let $x_k$ vary over $S_k^+$, and by duality between $S_k$ and $S_k'$, we deduce that $f_k \in (S_k')^+$; therefore, $(f_m) \in T_{\infty}^+$. 
\end{proof}

Combining these two propositions, we conclude the following duality theorem between $S_{\infty}$ and $T_{\infty}$. We write $S_{\infty}'$ for the ordered dual of $S_{\infty}$. 

\begin{thm}\label{AOUoi}
The dual pairing in \eqref{duality} induces an order isomorphism $\Gamma \colon T_{\infty} \to S_{\infty}' $ via $\Gamma \colon (f_k) \mapsto \inner{\cdot}{(f_k)}$. In particular, $(S_{\infty}', (S_{\infty}')^+) $ with $\Gamma( (\delta_k) )$ is an AOU space. 
\end{thm}

\subsection{Operator Systems}\label{dualityOS}

We now proceed to the case for \textbf{OS}. Let $(S_k, \phi_k)$ be an inductive sequence in \textbf{OS}, where each $\phi_k$ is a unital complete order embedding. Therefore, the inductive sequence $(S_k, \phi_k)$ in \textbf{OS} induces a projective sequence $(S_k', \phi_k')$, where $\phi_k'$ is surjective, in \textbf{OS}. 
Let $(S_{\infty}, \{ \phi_{k,\infty} \}_{k\in\mathbb{N}} )$ be the inductive limit of $(S_k, \phi_k)$ and $(T_{\infty}, \{ \psi_{\infty,k} \}_{k\in\mathbb{N}} )$ be the projective limit of $(S_k', \phi_k')$.

\begin{thm}\label{OScoi}
The duality defined in \eqref{duality} induces a complete order isomorphism $\Gamma \colon T_{\infty} \to S_{\infty}' $ via $\Gamma \colon (f_k) \mapsto \inner{\cdot}{(f_k)}$. Consequently, the matrix-ordered dual $(S_{\infty}', \{ M_n( S_{\infty}' )^+ \}_{n=1}^{\infty} )$, equipped with $\Gamma( (\delta_k) )$, is an operator system.  
\end{thm}

\begin{proof}
By Theorem \ref{AOUoi}, $\Gamma$ is an order isomorphism. At the matrix level, we shall prove that $[ (f_k)^{ij} ] \in M_n(T_{\infty})^+$ if and only if $F \colon S_{\infty} \to M_n$ by $F( \phi_{k,\infty}(x_k) ) := [ \inner{\phi_{k,\infty}(x_k)} {(f_k)^{ij}} ]$ is completely positive.

Suppose $[(f_k)^{ij}] \in M_n(T_{\infty})^+$. Then by definition of matrix-ordered dual and \eqref{projectiveOScone}, 
for each $k \in \mathbb{N}$, the map $F_k \colon S_k \to M_n$ by $F_k(x_k) = [ f_k^{ij}(x_k) ]$ is completely positive. By regarding $M_m(S_{\infty})$ as the inductive limit of $( M_m(S_k), \phi_k^{(m)} )$ in \textbf{AOU} and applying Remark \ref{inductiveAOU}, a matrix $[ \phi_{k,\infty}(x_k^{st}) ] \in M_m( S_{\infty} )^+$ if and only if for each $r > 0$, there exist $p \geq l > k$ with $[y_l^{st}] \in M_n( S_l)$ with $|| [ \phi_{l, p} ( y_l^{st} )] ||^p \to 0$, such that $r I_m \otimes e_p + [ \phi_{k,p} (x_k^{st}) ] + [\phi_{l,p} (y_l^{st}) ]  \in M_m(S_p)^+$. 
By a similar argument in the proof of Proposition \ref{AOUdualcones}, applying $F_p^{(m)}$ to this matrix yields that
	\begin{equation*}
		r I_m \otimes F_p(e_p) + [ f_p^{ij}( \phi_{k,p} (x_k^{st} ) )  ]  +  [ f_p^{ij}( \phi_{l,p} (y_l^{st} ) )  ] 		\in (M_m \otimes M_n) ^+.
	\end{equation*}
The third term vanishes as $p \to \infty$, so we have 
	\begin{equation*} 
		r \left| \left| \left[ (f_k)^{ij} \right] \right| \right|_{\max} I_m \otimes I_n + \left[\inner{\phi_{k,\infty}(x_k^{st}) }{(f_k)^{ij} } \right] \in (M_m \otimes M_n)^+, 
	\end{equation*}
for every $r > 0$. Therefore, $F^{(m)} ([ \phi_{k,\infty} (x_k^{st}) ] ) = [\inner{\phi_{k,\infty}(x_k^{st}) }{(f_k)^{ij} } ] \geq 0$, and $F$ is completely positive. 

Conversely, suppose $F$ is completely positive and $(X_k) = [ (x_k)^{st} ] \in M_m(S_k)^+$. By similar argument in $(\dagger)$, for large enough $p \geq k$, 
	\begin{align*}
			F_k^{(m)} (X_k) 	&=	[f_k^{ij} (x_k^{st} ) ]  = [f_p^{ij} ( \phi_{k,p} (x_k^{st}) ) ] \\	
													&=  \left[ \inner{ \phi_{k,\infty}^{(p)} (X_k) }{ (f_k)^{ij} } \right] = F^{(m)}( \phi_{k,\infty}^{(m)} (X_k) ),
	\end{align*}
where the last quantity is positive by hypothesis. Let $X_k$ vary over $M_n(S_k)^+$, then it follows that $[f_k^{ij}] \in M_n(S_k')^+$ for every $k \in \mathbb{N}$. Consequently, $[(f_k)^{ij}] \in M_n(T_{\infty})^+$; and $\Gamma$ is a complete order isomorphism.
\end{proof}

\section{Matrix-ordered duals of separable operator systems}\label{Section:dualityMain}
Our goal is to generalize Theorem \ref{DualFiniteDimension} to separable operator systems by duality between injective and projective limits of finite-dimensional operator systems. Henceforth, let $S$ be an operator system. We start by noting that the Archimedean property for order units and matrix order units in $S'$ is automatically inherited, due to the nature of weak*-topology  and the canonical identification $M_n(S') \cong M_n(S)'$. 

\begin{thm}\label{DualOrderUnitArch}
Let  $\delta \in S'$. The following are equivalent:
	\begin{enumerate}[label={\upshape(\roman*)}, align=left, widest=iii, leftmargin=*]
		\item $\delta$ is an order unit for $S'$.
		\item $\delta$ is a matrix order unit for $S'$.
		\item $\delta$ is an Archimedean order unit for $S'$.
		\item $\delta$ is an Archimedean matrix order unit for $S'$.
	\end{enumerate} 
\end{thm}

\begin{proof}
For (i) $\iff$ (iii), one direction is trivial. Assume (i) and suppose $f \in S'_h$ such that $r \delta + f \in (S')^+$ for each $r > 0$. Then for each $x \in S^+ \setminus \{0\}$, $\inner{r \delta + f }{x} \geq 0$. Letting $r \searrow 0$ implies that $f(x) \geq 0$, and $f \in (S')^+$; so $\delta$ is an Archimedean order unit.  
Now (i) $\iff$ (ii) follows from Proposition \ref{MOUunit}; and (ii) $\iff$ (iv) follows from (i) $\iff$ (iii) applied to $M_n(S') \cong M_n(S)'$. 
\end{proof}
	
\begin{remark}
We remark that a necessary condition of being an order unit for $S'$ is faithfulness. Indeed, suppose $\delta$ is an order unit for $S'$ and $x \in S^+$ such that $\delta(x) = 0$. Then for each $f \in (S')^+$, there exists $r > 0$ so that $r \delta - f \in (S')^+$. Thus, $(r \delta - f)(x) = - f(x) \geq 0$ implies that $f(x) = 0$ for all $f \in S'$. By duality, $x = 0$ and $\delta$ is faithful. The converse holds if $S$ in addition is reflexive as a normed space. 
\end{remark}

\begin{prop}\label{DualReflexive}
If $S$ is reflexive as a normed space and $\delta \in (S')^+$ is faithful over $S$, then $\delta$ is an order unit for $S'$. Moreover, in this case, $S'$ is an operator system. 
\end{prop}

\begin{proof}
Recall that  $S$ is reflexive if and only if the closed unit ball $B_1(S)$ of $S$ is weakly-compact. Let $f \in S'_h$ and consider the weakly-compact subset $K = B_1(S) \cap S^+$. Then $f$ attains its maximum on $K$. Also, $\delta$ being faithful asserts that the minimum of $\delta$ on $K$ is strictly positive. Hence, there exists $r > 0$ such that $r \delta - f  \geq 0$ on $S^+$; and $\delta$ is an order unit. The second statement now follows from Theorem \ref{DualOrderUnitArch}. 
\end{proof}

\begin{prop}\label{SeparableFaithful}
If $S$ is a separable operator system, then there exists a faithful linear functional $\delta$ over $S$. 
\end{prop}

\begin{proof}
Since $S$ is separable, $S'$ is weak*-metrizable. The state space $\mathfrak{S}(S)$ is weak*-compact in $S'$,  so it is weak*-separable. Let $\{ \delta_n \}$ be a weak*-dense sequence in $\mathfrak{S}(S)$ and define the functional $\delta \colon S \to [0,\infty)$ by $\delta(x) := \sum_n \frac{1}{2^n} \delta_n(x)$. Note that since $S'$ is a Banach space and $|| \delta_n || \leq 1$, $\delta$ is well-defined and $\delta \in \mathfrak{S}(S)$. 
To this end, suppose by contrary $\delta(x) = 0$ for some $x > 0$. Then $\delta_n(x) = 0$ for all $n \in \bb{N}$. By density, it implies that $f(x) = 0$ for all $f \in \mathfrak{S}(S)$. But $S'$ is the span of  $\mathfrak{S}(S)$, it follows that $f(x) = 0$ for all $f \in S'$, and $x = 0$. Therefore, $\delta$ is faithful over $S$. 
\end{proof}

Since finite-dimensional $S$ is both reflexive and separable, Theorem \ref{DualFiniteDimension} is a corollary of the above propositions. 
For infinite-dimensional $S$, to this end it remains to prove the existence of order unit for $S'$. 

\begin{prop}\label{DualCountableDim}
If $\dim(S)$ is countably infinite and $\delta$ is faithful over $S$, then $\delta$ is an Archimedean order unit for $S'$.
\end{prop}

\begin{proof}
Let $\{x_i=x_i^* \}_{i=1}^{\infty}$ with $x_1 = e$ be a Hamel basis for $S$ and let $S_k$ be the span of $x_i$, $i = 1, \dots, k$. Denote $\iota_k \colon S_k \to S_{k+1}$ the inclusion map and let $\delta_k$ be the restriction of $\delta$ to $S_k$. Note that $\iota_k'$ is the canonical complete order quotient map $q_k \colon f \mapsto f|_{S_k}$. 
It is clear that $S$ is the inductive limit of $(S_k, \iota_k)$ in \textbf{OS}. Since $\delta_k$ is faithful, by Proposition \ref{DualReflexive}, $(S_k', q_k)$ with $\delta_k$ is a projective sequence in \textbf{OS}. 
By Theorem \ref{OScoi}, $S'$ is completely order isomorphic to $\projectlim{OS} (S_k', q_k)$, whose Archimedean matrix order unit $( \delta_k )_{k\in\bb{N}}$ corresponds to $\delta \in S'$.
\end{proof}

For the separable case, we first consider dense operator subsystem $T$ of $S$. The following lemma must be well-known, but we could not find a precise reference.  

\begin{lemma}\label{DensityMatrixLevel}
Let $T$ be an operator subsystem of $S$. Then $T$ is dense in $S$ in the order norm topology if and only if for every $n \in \bb{N}$, $M_n(T)$ is dense in $M_n(S)$ in the order norm topology. 
\end{lemma}

\begin{proof}
One direction is trivial. Suppose $S \subset B(\cl{H})$ is a concrete operator system and $T$ is dense in $S$ in the order norm topology.  It suffices to show that $M_n(T)_h$ is dense in $M_n(S)_h$. Given $x \in M_n(S)_h$, by \cite[Lemma 3.7]{PTT}, decompose $x = \sum_{i=1}^N A_i \otimes x_i$, where $A_i \in M_n(S)_h$ and $x_i \in S_h$. Since $T_h = T \cap S_h$ is dense in $S_h$ in the subspace topology, for  $1\leq i \leq N$, there exists sequence $t_i^m \in T_h$ such that $|| x_i - t_i^m ||_h \to 0$. By \cite[Corollary 5.6]{PT} the order norm $|| \cdot ||_h$ is the operator norm $|| \cdot ||_{op}$ inherited from $B(\cl{H})$. Let $t = \sum_{i=1}^N A_i \otimes t_i \in M_n(T)_h$. Then in $M_n(S) \subset M_n(B(\cl{H})) \cong B( \cl{H}^{(n)})$, 
	\begin{align*}	
	|| t - x ||_h &= \left\|  \sum_{i=1}^N A_i \otimes (x_i - t_i^m) \right\|_{op} \leq \sum_{i=1}^N \lambda_i || x_i - t_i^m ||_{op}, 
	\end{align*}
where $\lambda_i = ||A_i||_{op}$. The last quantity goes to $0$ as $m \to \infty$, so $M_n(T)$ is dense in $M_n(S)$ in the order norm topology. 
\end{proof}

Suppose $T$ is a dense operator subsystem of separable $S$.  By a standard density argument, every $f \in T'$ has a unique extension $\tilde{f} \in S'$ such that $\tilde{f}|_{T} = f$. In fact, the map $f \mapsto \tilde{f}$ is an isometric isomorphism. 

\begin{prop}\label{DualDensity}
Let $S$ and $T$ be given as above. If $T'$ is an operator system with Archimedean matrix order unit $\delta$, then $S'$ is an operator system with  Archimedean matrix order unit $\tilde{\delta}$.  Furthermore, the map $f \mapsto \tilde{f}$ defines a unital complete order isomorphism between $T'$ and $S'$.  
\end{prop}

\begin{proof}
Given $\tilde{f} \in S'_h$, then its restriction to $T$ is $f \in T'_h$. By hypothesis, there exists $r > 0$ with $r \delta - f \in (T')^+$. We claim that $r \tilde{\delta} - \tilde{f} \in (S')^+$. If $x \in S^+$, then there exists sequence $t_m \in T^+ = T \cap S^+$ such that $t_m \to x$ under the order norm. Thus,  
	\begin{equation*}
		\inner{r\tilde{\delta} - \tilde{f}}{x} = \inner{r\tilde{\delta} - \tilde{f}}{ \lim_m t_m } = \lim_m \inner{r \delta - f }{t_m} \geq 0,
	\end{equation*}
where the second equality follows from the unique extension. Hence, $\tilde{\delta}$ is an order unit for $S'$. By Theorem \ref{DualOrderUnitArch}, $(S', \delta)$ is an operator system. 

Finally, by density and Lemma \ref{DensityMatrixLevel}, the map $f \mapsto \tilde{f}$ uniquely extends $CP(T, M_n)$ to $CP(S, M_n)$. Hence, it is unital completely positive with inverse $f \mapsto f|_{T}$. Therefore, $f \mapsto \tilde{f}$ is a unital complete order isomorphism.
\end{proof}

We conclude the main result of the paper. 

\begin{thm}\label{DualSeparable}
Let $S$ be a separable operator system. Then 
	\begin{enumerate}[label={\upshape(\roman*)}, align=left, widest=ii, leftmargin=*]
		\item there exists faithful $\delta \colon S \to \bb{C}$; and
		\item any faithful functional $\delta$ is an Archimedean matrix order unit for $S'$.
	\end{enumerate}
Consequently, $S'$ with any faithful state is an operator system.  
\end{thm}

\begin{proof}
The first statement is Proposition \ref{SeparableFaithful}. Suppose $\delta \in (S')^+$ is faithful and  $X$ is a countable dense subset of $S$. Define $T$ to be the span of $X$, $X^*$, and $e$ in $S$. Then $T$ is a dense operator subsystem of $S$ with countable dimension. By Proposition  \ref{DualCountableDim}, $(T', \delta|_T)$ is an operator system. The result now follows from Proposition \ref{DualDensity}. 
\end{proof}

\section*{Acknowledgement}
The post-doctoral fellowship of Ng at the Hong Kong Polytechnic University is supported by the PolyU central research grant G-YBKR, the HK RGC grant PolyU 502512, and the AMSS-PolyU Joint Research Institute.


\providecommand{\bysame}{\leavevmode\hbox to3em{\hrulefill}\thinspace}
\providecommand{\MR}{\relax\ifhmode\unskip\space\fi MR }
\providecommand{\MRhref}[2]{%
  \href{http://www.ams.org/mathscinet-getitem?mr=#1}{#2}
}
\providecommand{\href}[2]{#2}

\end{document}